\documentclass{amsart}
\usepackage[utf8]{inputenc}
\usepackage[T1]{fontenc}
\usepackage{lmodern}

\usepackage{amssymb}
\usepackage{amsfonts}
\usepackage{stackrel}

\usepackage[lite]{amsrefs}

\renewcommand*{\PrintDOI}[1]{\href{http://dx.doi.org/\detokenize{#1}}{doi: \detokenize{#1}}}

\BibSpec{thesis}{%
  +{}  {\PrintAuthors}                {author}
  +{,} { \textit}                     {title}
  +{:} { \textit}                     {subtitle}
  +{,} { \PrintThesisType}            {type}
  +{,} { }                            {organization}
  +{,} { }                            {address}
  +{,} { \PrintDateB}                 {date}
  +{,} { \eprint}                     {eprint}
  +{,} { }                            {status}
  +{}  { \parenthesize}               {language}
  +{}  { \PrintTranslation}           {translation}
  +{;} { \PrintReprint}               {reprint}
  +{.} { }                            {note}
  +{.} {}                             {transition}
  +{} { \PrintDOI}                   {doi}
  +{}  {\SentenceSpace \PrintReviews} {review}
}

\usepackage{mathtools}
\usepackage{enumitem}
\usepackage{microtype}
\usepackage{tikz-cd}

\usepackage[pdftitle={A universal coefficient theorem for actions of
  finite cyclic groups of square-free order on C*-algebras},
pdfauthor={Ralf Meyer, George Nadareishvili},
pdfsubject={Mathematics}
]{hyperref}

\setlist[enumerate,1]{label=\textup{(\arabic*)}}
\setlist[enumerate,2]{label=\textup{(\alph*)}}

\numberwithin{equation}{section}

\newtheorem{theorem}{Theorem}[section]
\newtheorem{lemma}[theorem]{Lemma}
\newtheorem{proposition}[theorem]{Proposition}
\newtheorem{corollary}[theorem]{Corollary}

\theoremstyle{definition}
\newtheorem{definition}[theorem]{Definition}

\theoremstyle{remark}
\newtheorem{remark}[theorem]{Remark}
\newtheorem{example}[theorem]{Example}

\newcommand*{\Cst}{\textup C^*}
\newcommand*{\Cont}{\textup C}
\newcommand*{\nb}{\nobreakdash}
\newcommand{\into}{\rightarrowtail}
\newcommand{\prto}{\twoheadrightarrow}

\newcommand{\op}{\mathrm{op}}
\newcommand{\ext}{\mathrm{ext}}
\newcommand{\C}{\mathbb{C}}
\newcommand{\Q}{\mathbb{Q}}
\newcommand{\R}{\mathbb{R}}
\newcommand{\K}{\mathrm{K}}
\newcommand{\Z}{\mathbb{Z}}
\newcommand{\KK}{\mathrm{KK}}
\newcommand{\ideal}{\mathfrak{I}}

\DeclareMathOperator{\im}{im} 
\DeclareMathOperator{\Hom}{Hom} 
\DeclareMathOperator{\Ext}{Ext} 
\DeclarePairedDelimiter{\abs}{\lvert}{\rvert}
\DeclarePairedDelimiterX{\setgiven}[2]{\{}{\}}{#1\,{:}\,\mathopen{}#2}

\newcommand{\Ga}[2]{\alpha_{#1#2}}
\newcommand{\Gt}[1]{t_{#1}}
\newcommand{\Gs}[1]{s_{#1}}
\newcommand{\Gu}[1]{1_{#1}}

\newcommand{\Kring}{\mathfrak{K}}
\newcommand{\MG}{Q}
\newcommand{\DmG}{Q}
\newcommand{\Mp}{M}

\newcommand{\Mack}{\mathcal M}

\newcommand*{\defeq}{\mathrel{:=}}

\DeclareMathOperator{\Res}{Res}
\DeclareMathOperator{\Ind}{Ind}

\title[A UCT for actions of finite cyclic groups of square-free order]{A
  universal coefficient theorem for actions of finite cyclic groups of
  square-free order on $\Cst$\nb-algebras}

\author{Ralf Meyer}
\email{rmeyer2@uni-goettingen.de}
\address{Mathematisches Institut\\
  Universit\"at G\"ottingen\\Bunsenstra\ss e 3--5\\
  37073 G\"ottingen\\Germany}

\author{George Nadareishvili}
\email{giorgi.nadareishvili@kiu.edu.ge}
\address{School of Mathematics\\
Kutaisi International University\\Akhalgazrdoba Ave.\ Lane 5/7\\
4600 Kutaisi\\Georgia
}
\thanks{George Nadareishvili was supported by Shota Rustaveli National Science Foundation of Georgia, FR-25-10267.}


\begin{document}

\begin{abstract}
  We prove a Universal Coefficient Theorem for objects in the
  bootstrap class in the equivariant Kasparov category for a finite
  cyclic group of square-free order.
\end{abstract}

\maketitle

\section{Introduction}

The classical Universal Coefficient Theorem (UCT) by Rosenberg and
Schochet~\cite{Rosenberg-Schochet:Kunneth} is a fundamental tool in
the study of \(\Cst\)\nb-algebras.
The theorem uses a short exact sequence to compute bivariant
\(\K\)\nb-theory groups of separable \(\Cst\)\nb-algebras in the
bootstrap class in terms of their \(\K\)\nb-theory.
It implies that two \(\Cst\)\nb-algebras in the bootstrap class are
\(\KK\)-equivalent once they have isomorphic \(\K\)\nb-theory.
This is a step towards the classification of Kirchberg algebras.

The UCT by Rosenberg--Schochet has been generalised to various
situations involving \(\Cst\)\nb-algebras with extra structure.
Meyer and Nest~\cite{Meyer-Nest:Filtrated_K} established a UCT for
\(\Cst\)\nb-algebras over a finite sober topological space~\(X\) using
filtrated \(\K\)\nb-theory.
This led to new classification results as well.
In the equivariant setting, Manuel K\"ohler~\cite{Koehler:Thesis}
established a remarkable UCT for \(\Cst\)\nb-algebras with an action
of a cyclic group of prime order.
This result has since been applied by
Meyer~\cite{Meyer:Actions_Kirchberg} to the classification of actions
of~\(\Z/p\) on Kirchberg algebras up to equivariant
\(\KK\)\nb-equivalence.

The aim of the present article is a generalisation of K\"ohler's UCT
for actions of finite cyclic groups~\(G\) of square-free order.  We
rely on the fact that the equivariant bootstrap
class~\(\mathfrak{B}^G\), which is defined to consist of all
\(G\)\nb-\(\Cst\)\nb-algebras that are \(\KK^G\)\nb-equivalent to an
action on a Type~I \(\Cst\)\nb-algebra, is generated by \(\Cont(G/H)\)
for cyclic subgroups \(H \subseteq G\).  This was recently shown
in~\cite{Meyer-Nadareishvili:UCT_actions}, using results
from~\cite{Arano-Kubota:Atiyah-Segal}.

To generalise K\"ohler's invariant, we use a larger generating
set~\(\mathfrak{C}^G\) for~\(\mathfrak{B}^G\), which consists of the
tensor products of K\"ohler's generators for the prime\nb-order
subgroups.  We prove that the resulting \(\Z/2\)-graded
category ring, \(\Kring_G\), is isomorphic to the tensor product of
individual rings~\(\Kring_{p_i}\) computed by K\"ohler.  Our main
result is:

\begin{theorem}
  \label{the:Eqv_UCT}
  Let~\(G\) be a finite cyclic group of square-free order.
  Let \(\mathfrak{B}^G\subseteq \KK^G\) be the equivariant bootstrap
  class.
  There is a stable homological functor~\(U_{\mathfrak{C}^G}\) from
  \(\KK^G\) to the category of \(\Z/2\)-graded exact countable modules
  over a certain ring~\(\Kring_G\), such that for any
  \(A\in\mathfrak{B}^G\) and \(C\in \KK^G\), there is a natural short
  exact sequence
  \[
    \Ext^1_{\Kring_G}\bigl(U_{\mathfrak{C}^G}(\Sigma A),
    U_{\mathfrak{C}^G}(C)\bigr) \into
    \KK^G_*(A,C) \prto
    \Hom_{\Kring_G}\bigl(U_{\mathfrak{C}^G}(A),
    U_{\mathfrak{C}^G}(C)\bigr).
  \]
  If \(C\in\mathfrak{B}^G\) as well, then every isomorphism
  \(U_{\mathfrak{C}^G}(A) \cong U_{\mathfrak{C}^G}(C)\) of
  \(\Z/2\)\nb-graded \(\Kring_G\)\nb-modules lifts to a
  \(\KK^G\)-equivalence in \(\KK^G_0(A,C)\).  In particular, \(A\)
  and~\(C\) are \(\KK^G\)-equivalent if and only if
  \(U_{\mathfrak{C}^G}(A) \cong U_{\mathfrak{C}^G}(C)\) as
  \(\Z/2\)\nb-graded \(\Kring_G\)\nb-modules.
\end{theorem}

Our proof method also describes the range of the invariant: these are
precisely the \(\Z/2\)-graded modules over~\(\Kring_G\) that are
countable and exact.
It was shown in~\cite{Meyer:Actions_Kirchberg} that any
\(\KK^G\)-equivalence class in the equivariant bootstrap class is
represented by a pointwise outer action on a Kirchberg algebra, still
in the bootstrap class.
Gabe and Szab\'o~\cite{Gabe-Szabo:Dynamical_Kirchberg} have shown that
two such actions are cocycle conjugate if and only if they are
\(\KK^G\)-equivalent.
Thus we get a classification of pointwise outer actions of~\(G\) on
Kirchberg algebras in the bootstrap class.

We start with some preliminaries on homological algebra in equivariant
KK-theory and on K\"ohler's UCT in Section~\ref{sec:preliminaries}.
Section~\ref{sec:KK_product} contains a computation of certain
equivariant \(\KK\)-groups when the group is a product and both
\(\Cst\)\nb-algebras are tensor products; this is needed to compute
the ring~\(\Kring_G\).
Section~\ref{sec:extension} describes the
functor~\(U_{\mathfrak{C}^G}\) and contains the proof of
Theorem~\ref{the:Eqv_UCT}.
The key lemma says that any module over~\(\Kring_G\) in the range of
the invariant has a projective resolution of length~\(1\).
We reduce this to the analogous result for cyclic groups of prime
order using that~\(\Kring_G\) is a tensor product of rings for the
prime order factors of~\(G\).

\section{Some preliminaries}
\label{sec:preliminaries}

Let~\(G\) be a second countable, locally compact group.  The Kasparov
category \(\KK^G\) of separable \(G\)\nb-\(\Cst\)-algebras is a
symmetric monoidal triangulated category.  Its tensor product is given
by the spatial tensor product of \(\Cst\)\nb-algebras equipped with
the diagonal \(G\)\nb-action.  Exact triangles arise either from
mapping cones of equivariant \(^*\)-homomorphisms, or from extensions
of \(G\)\nb-\(\Cst\)-algebras that admit a \(G\)\nb-equivariant,
completely positive, contractive section (see the Appendix
of~\cite{Meyer-Nest:BC}).  The suspension functor is defined by
\(\Sigma \defeq \Cont_0(\R)\otimes {-}\).  By Bott periodicity, this
functor is an involutive equivalence.  Moreover, \(\KK^G\) admits
countable coproducts induced by \(\Cont_0\)-direct sums, denoted
by~\(\oplus\).

\subsection{Homological algebra in \texorpdfstring{$\KK^G$}{KKG}}

The aim of this article is to extend the Universal Coefficient Theorem
of~\cite{Koehler:Thesis} to finite cyclic groups of square-free order.
The framework for this is homological algebra in triangulated
categories.  Therefore, we recall some basic facts from this theory.
The situation that we will need is a special case of the following more
general setup.

Let~\(\mathfrak{T}\) be a triangulated category with countable
coproducts.  Let~\(\Sigma\) denote its suspension.  An object
\(C\in\mathfrak T\) is called \emph{\(\aleph_1\)\nb-compact} if the
representable functor
\(\mathfrak{T}(C,{-})\colon \mathfrak T\to \mathfrak{Ab}\) to the
category of abelian groups commutes with countable coproducts.  Now
let~\(\mathfrak{C}\) be an at most countable set of such objects
in~\(\mathfrak{T}\), and assume moreover that
\(\mathfrak{T}_n(C,A) \defeq \mathfrak{T}(\Sigma^n C, A)\) is
countable for all \(A \in \mathfrak{T}\) and \(n \in \Z\).
Let \(\mathfrak{Ab}^{\Z}\) denote the abelian category of
\(\Z\)\nb-graded abelian groups, equipped with the suspension
homomorphism that shifts degrees.  Define the functor
\[
  F_{\mathfrak{C}}\colon\mathfrak{T} \to
  \prod_{C\in \mathfrak{C}}\mathfrak{Ab}^{\Z}, \qquad
  A \mapsto \bigl(\mathfrak{T}_n(C,A)\bigr)
  _{C \in \mathfrak{C},n\in\Z}.
\]
This intertwines the suspensions in~\(\mathfrak T\) and
\(\mathfrak{Ab}^{\Z}\).
We call a category with a fixed (suspension) automorphism
\emph{stable}.
We call a functor between stable categories \emph{stable} when it
intertwines the suspensions.
The kernel of~\(F_{\mathfrak{C}}\) on morphisms,
\(\ideal_{\mathfrak{C}}\), is the prototypical example of a
stable homological ideal in a triangulated category.
This ideal specifies how to do relative homological algebra
in~\(\mathfrak T\).

Let \(\langle \mathfrak{C} \rangle \subseteq \mathfrak T\) denote the
smallest triangulated subcategory of~\(\mathfrak T\) that is closed
under countable coproducts and contains~\(\mathfrak C\).
This subcategory will play the role of the \emph{bootstrap category}
in \(\KK\).

A stable homological functor \(H\colon \mathfrak T\to \mathfrak A\)
is called \emph{\(\ideal\)\nb-exact} if it vanishes on the
ideal~\(\ideal\).
An \(\ideal\)\nb-exact stable homological functor
\(U\colon \mathfrak T\to \mathfrak{A}_\ideal\) is called
\emph{universal} if any \(\ideal\)\nb-exact stable homological functor
\(H\colon \mathfrak T\to \mathfrak A\) factors uniquely as
\(\bar{H}\circ U\), for a stable exact functor
\(\bar{H}\colon \mathfrak{A}_\ideal\to \mathfrak A\).
Such a universal functor closely relates the
\(\ideal\)\nb-relative homological algebra in~\(\mathfrak{T}\) to the
homological algebra in the abelian category~\(\mathfrak{A}_\ideal\).
In particular, the relative derived functors in~\(\mathfrak{T}\) are
identified with those in~\(\mathfrak{A}_\ideal\) after composition
with~\(U\).

We now describe the universal \(\ideal_{\mathfrak{C}}\)\nb-exact
stable homological functor.
Let~\(\mathfrak{C}\) also denote the \(\Z\)\nb-graded
pre-additive category with objects~\(\mathfrak{C}\) and morphisms
\(\bigoplus_{n \in \Z} \mathfrak{T}_n(A,B)\) for \(A,B \in
\mathfrak{C}\).
A \emph{right \(\mathfrak{C}\)-module} is defined as a contravariant
stable additive functor \(\mathfrak{C}^\op \to
\mathfrak{Ab}^{\Z}\).
These modules form a stable abelian category
\(\mathfrak{Mod}(\mathfrak{C}^\op)\) with direct sums and enough
projective objects.
The subcategory of countable modules is denoted by
\(\mathfrak{Mod}(\mathfrak{C}^\op)_{\aleph_1}\).
Equipping \(\bigl(\mathfrak{T}_n(C,A)\bigr)_{n \in \Z}\) with
the right \(\mathfrak{C}\)\nb-module structure induced by composition
in~\(\mathfrak{T}\), we enrich~\(F_{\mathfrak{C}}\) to a functor
\[
  U_{\mathfrak{C}} \colon \mathfrak{T} \longrightarrow
  \mathfrak{Mod}(\mathfrak{C}^\op)_{\aleph_1}.
\]
The category \(\mathfrak{Mod}(\mathfrak{C}^\op)_{\aleph_1}\)
is isomorphic to the category of left modules over the category ring
of~\(\mathfrak{C}^\op\).

The following is the abstract Universal Coefficient Theorem (UCT):

\begin{theorem}
  \label{thm:truct}
  Let~\(\mathfrak T\) be a triangulated category with countable
  coproducts and let \(\mathfrak C \subseteq \mathfrak{T}\) be a set
  of \(\aleph_1\)\nb-compact objects.  The universal
  \(\mathfrak{I}_{\mathfrak{C}}\)\nb-exact stable homological functor
  is~\(U_\mathfrak{C}\).

  Let \(A\in \langle \mathfrak C\rangle\) and \(B \in \mathfrak T\).
  If the \(\mathfrak{C}\)-module \(U_{\mathfrak{C}}(A)\) has a
  projective resolution of length~\(1\), then there is a natural short
  exact sequence
  \[
    \Ext_{{\mathfrak C}}^1
    \bigl(U_{\mathfrak{C}}(\Sigma A),U_{\mathfrak{C}}(B)\bigr)
    \into \mathfrak{T}(A,B)
    \prto \Hom_{{\mathfrak C}}
    \bigl(U_{\mathfrak{C}}(A),U_{\mathfrak{C}}(B)\bigr).
  \]
\end{theorem}

\begin{proof}
  These statements are contained in
  \cite{Meyer-Nest:Homology_in_KK}*{Theorem~4.4}.
\end{proof}

When \(\mathfrak{T}\) is an equivariant Kasparov category, the
suspension functor is an involutive equivalence.
Hence the \(\Z\)\nb-graded modules above become \(\Z/2\)\nb-graded.

\begin{example}
  \label{exa:nonequiv}
  Let~\(G\) be a trivial group.
  Then \(\mathfrak{C} = \{\C\}\)
  generates the well-known bootstrap class and the functor
  \(U_{\mathfrak{C}}\) is topological \(\K\)\nb-theory, regarded as a
  functor to the stable abelian category of countable
  \(\Z/2\)-graded abelian groups.  The latter has global
  homological dimension~\(1\).  Hence, Theorem~\ref{thm:truct}
  applies.  It gives the Universal Coefficient Theorem of Rosenberg
  and Schochet~\cite{Rosenberg-Schochet:Kunneth}.
\end{example}

\subsection{The equivariant bootstrap class}
\label{sec:equivariant_bootstrap}

The bootstrap class in Example~\ref{exa:nonequiv} has several
equivalent descriptions.
Some of these still work in the equivariant case.

\begin{definition}
  Let~\(G\) be a finite group.
  The \emph{equivariant bootstrap class}~\(\mathfrak B^G\) is defined
  as the class of \(G\)\nb-\(\Cst\)\nb-algebras that are
  \(\KK^G\)\nb-equivalent to a \(G\)\nb-action on a Type~I
  \(\Cst\)\nb-algebra.
\end{definition}

The bootstrap class may also be described as the localising
subcategory generated by the collection of actions of~\(G\) on
finite-dimensional \(\Cst\)\nb-algebras.
This generating set, however, is rather large and clearly has some
redundancies.
It is useful to find a set of generators that is irredundant in the
sense that no proper subset is again a generating set.
For a finite group~\(G\), a smaller generating set was found
in~\cite{Meyer-Nadareishvili:UCT_actions}, and we are going to show
that it is irredundant in the above sense in
Proposition~\ref{pro:necessary_gen}.
First, we describe this generating set:

\begin{theorem}[\cite{Meyer-Nadareishvili:UCT_actions}*{Corollary~3.3}]
  \label{thm:generators}
  Let~\(G\) be a finite group.  Choose one representative for each
  conjugacy class of cyclic subgroups and let~\(\mathfrak{C}\) be the
  set of \(G\)\nb-\(\Cst\)\nb-algebras \(\Cont(G/H)\) for these
  representatives of cyclic subgroups \(H\subseteq G\).  The bootstrap
  class is the smallest localising subcategory that
  contains~\(\mathfrak{C}\).
\end{theorem}

\begin{lemma}
  \label{lem:K_does_not_detect}
  Let~\(G\) be a finite group and \(L\subseteq G\) a cyclic
  subgroup.  There is \(A\in \mathfrak{B}^G\) such that
  \(\K_*^L(A)\neq 0\) and \(\K_*^H(A)= 0\) for
  subgroups \(H\subseteq G\) that are not conjugate to~\(L\).
\end{lemma}

\begin{proof}
  The object~\(A\) will, in fact, be ``rational'', that is, it belongs
  to the localisation~\(\mathfrak{B}^G_\Q\)
  of~\(\mathfrak{B}^G\) at the rational numbers.
  It is shown in~\cite{Bouc-DellAmbrogio-Martos:Splitting} that there
  is an equivalence from~\(\mathfrak{B}^G_\Q\) to the category
  of Mackey modules over the rationalised representation Green ring
  of~\(G\) that maps any object~\(A\) to the family \(\K_*^L(A)\) for
  all \(L\subseteq G\) with its canonical Mackey module structure.
  Therefore, an object with the desired properties exists
  in~\(\mathfrak{B}^G_\Q\) if and only if there is a Mackey
  module~\(\Mack\) over the rationalised representation Green ring
  of~\(G\) with the property that \(\Mack(L)\neq 0\) and \(\Mack(H)= 0\)
  for all subgroups \(H\subseteq G\) not conjugate to~\(L\).
  Such a Mackey module exists by
  \cite{Thevenaz-Webb:The_structure_of_Mackey_functors}*{Example~6.7}
  and
  \cite{Thevenaz-Webb:The_structure_of_Mackey_functors}*{Lemma~6.4}.
\end{proof}

A similar argument does not work for non-cyclic subgroups because
Artin’s Induction Theorem is used
in~\cite{Thevenaz-Webb:The_structure_of_Mackey_functors}*{Example~6.7}.

\begin{proposition}
  \label{pro:necessary_gen}
  Let~\(G\) be a finite group and \(L\subseteq G\) a cyclic subgroup.
  Then \(\langle \Cont(G/H)\mid H\subseteq G \text{ cyclic, not
    conjugate to } L \rangle\) is strictly smaller than
  \(\mathfrak{B}^{G}\).
\end{proposition}

\begin{proof}
  For a subgroup~\(H\), let \(\Res^G_H\colon \KK^G\to\KK^H\) be the
  functor that restricts a \(G\)\nb-action to~\(H\).  Recall that the
  induction functor \(\Ind^G_H\colon \KK^H\to\KK^G\) is left adjoint
  to it (see~\cite{Meyer-Nest:BC}) and that
  \(\Cont(G/H) = \Ind_H^G \C\).  Let~\(A\) be any
  \(G\)\nb-\(\Cst\)\nb-algebra.  Then
  \[
    \KK^G_*(\Cont(G/H),A)
    \cong \KK^H_*(\C,A)
    \cong \K_*^H(A).
  \]
  Therefore, the object~\(A\) from Lemma~\ref{lem:K_does_not_detect}
  satisfies \(\KK^G_*(\Cont(G/L),A) \neq 0\) and
  \(\KK^G_*(\Cont(G/H),A) = 0\) for all other cyclic subgroups
  \(H\subseteq G\) that are not conjugate to~\(L\).
  Then \(\KK^G_*(D,A)=0\) for all objects in the localising
  subcategory generated by \(\Cont(G/H)\) for \(H\subseteq G\) not
  conjugate to~\(L\).
  So~\(A\) cannot belong to this localising subcategory.
\end{proof}

\subsection{K\"ohler's invariant}
\label{sec:kohler}

Let~\(p\) be a prime number and let \(G = \Z/p\).  Let~\(C_p\)
denote the mapping cone of the \(\Z/p\)\nb-equivariant unital
embedding \(\C\to \Cont(\Z/p)\).  Define
\[
  \mathfrak C^{\Z/p} = \left\{ C_p, \C, \Cont(\Z/p) \right\}.
\]
Let~\(\Kring_p\) be the \(\Z/2\)\nb-graded category ring of
\(({\mathfrak C^{\Z/p}})^\op\); that is,
\[
  \Kring_p =
  \KK^{\Z/p}\bigl(C_p\oplus\C\oplus\Cont(\Z/p),C_p\oplus
  \C\oplus \Cont(\Z/p)\bigr)^\op.
\]
The ring~\(\Kring_p\) was computed by Manuel Köhler
in~\cite{Koehler:Thesis}.  He showed that every object in the range of
the invariant~\(U_{\mathfrak{C}^{\Z/p}}\) admits a projective
resolution of length~\(1\).  Using this, he established an equivariant
Universal Coefficient Theorem for \(\KK^G\) (Theorem~\ref{the:range_U}
below).  Since we will be concerned with specific properties
of~\(\Kring_p\), we recall its presentation in terms of generators and
relations (see \cites{Koehler:Thesis, Meyer:Actions_Kirchberg} for
more details).

\begin{theorem}[\cite{Meyer:Actions_Kirchberg}*{Theorem~5.10}]
  \label{the:generators_and_relations}
  The ring~\(\Kring_p\) is the universal ring generated by elements
  \(\Gu{j}\) for \(j=0,1,2\) and~\(\Ga{j}{k}\) for \(0\le j,k \le 2\)
  with \(j\neq k\) with the relations
  \begin{align*}
    \Gu{j}\Gu{k}&= \delta_{j,k} \Gu{j}
    \qquad \text{for }j,k\in\{0,1,2\},\\
    \Gu{0}+\Gu{1}+\Gu{2}&= 1.\\
    \Gu{j} \Ga{j}{k} \Gu{k}&= \Ga{j}{k},\\
    \Ga{j}{k} \Ga{k}{m}&= 0\qquad \text{if } \{j,k,m\}=\{0,1,2\}, \\
    \Ga{0}{1}\Ga{1}{0}&= N(\Gu{0}- \Ga{0}{2}\Ga{2}{0}),\\
    \Ga{1}{0}\Ga{0}{1}&= N(\Gu{1}- \Ga{1}{2}\Ga{2}{1}),\\
    p\cdot \Gu{2} &= N(\Gu{2}- \Ga{2}{0}\Ga{0}{2})
        +  N(\Gu{2}- \Ga{2}{1}\Ga{1}{2}),
  \end{align*}
  where \(N(t) \defeq 1 + t + \dotsb + t^{p-1}\).
  The \(\Z/2\)\nb-grading on~\(\Kring_p\) is such that \(\Ga{1}{2}\)
  and \(\Ga{2}{1}\) are odd and all other generators are even.
\end{theorem}

This theorem is proven in the above form
in~\cite{Meyer:Actions_Kirchberg}, following the computations
in~\cite{Koehler:Thesis}.  It is very convenient to name the following
elements of~\(\Kring_p\):
\[
  \begin{aligned}
    \Gt{0}&\defeq \Gu{0}- \Ga{0}{2} \Ga{2}{0},&\qquad
    \Gs{1}&\defeq \Gu{1}- \Ga{1}{2} \Ga{2}{1},\\
    \Gt{2}&\defeq \Gu{2}- \Ga{2}{0} \Ga{0}{2},&\qquad
    \Gs{2}&\defeq \Gu{2}- \Ga{2}{1} \Ga{1}{2}.
  \end{aligned}
\]
The elements \(\Gt{j}\) and~\(\Gs{j}\) are even.  The relations in
Theorem~\ref{the:generators_and_relations} say that
\begin{align}
  \label{eq:alpha_010_N}
  \Ga{0}{1} \Ga{1}{0}&= N(\Gt{0}),\\
  \label{eq:alpha_101_N}
  \Ga{1}{0} \Ga{0}{1}&= N(\Gs{1}),\\
  \label{eq:alpha_2_N}
  p\cdot \Gu{2} &= N(\Gt{2}) +  N(\Gs{2}).
\end{align}
The domain and codomain projections of the elements above are shown in
Figure~\ref{fig:K_p}.
\begin{figure}[tbp] \centering
  \begin{tikzcd}[row sep = 5em, column sep = 2.5em]
    &1 \ar[loop above, "1_1{,}s_1"] \ar[dr, shift right, "\Ga{2}{1}"']
    \ar[dl, shift right, "\Ga{0}{1}"'] &\\
    0 \ar[loop left, "1_0{,}t_0"] \ar[rr, shift right,"\Ga{2}{0}"']
    \ar[ur, shift right, "\Ga{1}{0}"']&&
    2 \ar[loop right, "1_2{,}t_2{,}s_2"] \ar[ul, shift right,
    "\Ga{1}{2}"']
    \ar[ll, shift right, "\Ga{0}{2}"']
  \end{tikzcd}
  \caption{The ring~\(\Kring_p\) can be viewed as the
    \(\Z/2\)\nb-graded path ring of a graph with three vertices and the
    six edges~\(\alpha_{ij}\), modulo the relations described in
    Theorem~\ref{the:generators_and_relations}.}
  \label{fig:K_p}
\end{figure}
The definitions of \(\Gt{j}\) and~\(\Gs{j}\) say that
\begin{align}
  \label{eq:alpha_020}
  \Ga{0}{2} \Ga{2}{0}&=  \Gu{0}- \Gt{0},\\
  \label{eq:alpha_121}
  \Ga{1}{2} \Ga{2}{1}&=  \Gu{1}- \Gs{1},\\
  \label{eq:alpha_202}
  \Ga{2}{0} \Ga{0}{2}&=  \Gu{2}- \Gt{2},\\
  \label{eq:alpha_212}
  \Ga{2}{1} \Ga{1}{2}&=  \Gu{2}- \Gs{2}.
\end{align}
Equations \eqref{eq:alpha_010_N}, \eqref{eq:alpha_101_N} and
\eqref{eq:alpha_020}--\eqref{eq:alpha_212} express the products
\(\Ga{j}{k} \Ga{k}{j}\) for all \(j, k \in \{0,1,2\}\) with
\(j \neq k\), as polynomials in \(\Gt{j}\) or~\(\Gs{j}\).  All other
products of two \(\alpha\)\nb-generators are zero.  The following
relations follow from those in
Theorem~\ref{the:generators_and_relations} and will be used in later
proofs:
{\allowdisplaybreaks[1]  
  \begin{align}
    \label{eq:tp}
    \Gt{j}^p &=1\qquad \text{for }j=0,2,\\
    \label{eq:sp}
    \Gs{j}^p &=1\qquad \text{for }j=1,2,\\
    \label{eq:ta_01}%
    \Gt{0} \Ga{0}{1}&= \Ga{0}{1},\\
    \label{eq:ta_20}%
    \Gt{2} \Ga{2}{0}&= \Ga{2}{0} \Gt{0},\\
    \label{eq:at_01}%
    \Ga{1}{0} \Gt{0}&= \Ga{1}{0},\\
    \label{eq:sa_10}%
    \Gs{1} \Ga{1}{0}&= \Ga{1}{0},\\
    \label{eq:sa_21}%
    \Gs{2}\Ga{2}{1}&= \Ga{2}{1} \Gs{1},\\
    \label{eq:as_01}%
    \Ga{0}{1} \Gs{1}&= \Ga{0}{1},\\
    \label{eq:Na_02}%
    N(\Gt{0}) \Ga{0}{2}&= 0,\\
    \label{eq:aN_20}%
    \Ga{2}{0} N(\Gt{0}) &= 0,\\
    \label{eq:Na_12}%
    N(\Gs{1})\Ga{1}{2}&= 0,\\
    \label{eq:aN_21}%
    \Ga{2}{1} N(\Gs{1}) &= 0.
  \end{align}
}%
We also need to characterise the range of the
invariant~\(U_{\mathfrak{C}^{\Z/p}}\).
Given a left \(\Kring_p\)\nb-module~\(\Mp\), write \(\Mp = \Mp_0
\oplus \Mp_1 \oplus \Mp_2\), corresponding to the decomposition
of~\(\Kring_p\).
Let~\(\Ga{j}{k}^\Mp\) for \(j\neq k\) also denote the map
\(\Mp_k \to \Mp_j\), \(y\mapsto \Ga{j}{k} y\).

\begin{definition}
  \label{def:exact_K-module}
  A left \(\Kring_p\)\nb-module~\(\Mp\) is \emph{exact} if the two
  sequences of abelian groups clockwise and counterclockwise around
  the following triangle are exact:
  \[
    \begin{tikzcd}[row sep =5em, column sep =2.5em]
      &\Mp_1 \ar[dr, shift right, "\Ga{2}{1}^\Mp"'] \ar[dl, shift right, "\Ga{0}{1}^\Mp"']&\\
      \Mp_0 \ar[rr, shift right, "\Ga{2}{0}^\Mp"'] \ar[ur, shift right, "\Ga{1}{0}^\Mp"']&&
      \Mp_2 \ar[ul, shift right, "\Ga{1}{2}^\Mp"'] \ar[ll, shift right, "\Ga{0}{2}^\Mp"']
    \end{tikzcd}
  \]
\end{definition}

\begin{theorem}[\cite{Koehler:Thesis}]
  \label{the:range_U}
  Let~\(\Mp\) be a countable \(\Z/2\)\nb-graded left
  \(\Kring_p\)\nb-module.  The following are equivalent:
  \begin{enumerate}[label=\textup{(\arabic*)}]
  \item \label{en:range_U_BG}%
    \(\Mp = U_{\mathfrak{C}^{\Z/p}}(A)\) for some~\(A\)
    in~\(\mathfrak{B}^{\Z/p}\);
  \item \label{en:range_U_KKG}%
    \(\Mp = U_{\mathfrak{C}^{\Z/p}}(A)\) for some~\(A\)
    in~\(\KK^{\Z/p}\);
  \item \label{en:range_U_exact}%
    \(\Mp\) is exact;
  \item \label{en:range_U_proj}%
    \(\Mp\) has a projective \(\Kring_p\)\nb-module resolution of
    length~\(1\).
  \end{enumerate}
\end{theorem}

This implies the following Universal Coefficient Theorem for
\(\KK^{\Z/p}\):

\begin{theorem}[\cite{Koehler:Thesis}]
  \label{the:Koehler_invariant_UCT_classifies}
  Let \(A\) and~\(C\) be separable \(\Z/p\)\nb-\(\Cst\)-algebras with
  \(A\in\mathfrak{B}^{\Z/p}\).  Then there is a natural short exact
  sequence
  \[
    \Ext^1_{\Kring_p}\bigl(U_{\mathfrak{C}^{\Z/p}}(\Sigma A),
    U_{\mathfrak{C}^{\Z/p}}(C)\bigr) \into
    \KK^{\Z/p}_*(A,C) \prto
    \Hom_{\Kring_p}\bigl(U_{\mathfrak{C}^{\Z/p}}(A),U_{\mathfrak{C}^{\Z/p}}(C)\bigr).
  \]
  If \(A,C\in\mathfrak{B}^{\Z/p}\), then every isomorphism
  \(U_{\mathfrak{C}^{\Z/p}}(A) \cong U_{\mathfrak{C}^{\Z/p}}(C)\) of
  \(\Z/2\)\nb-graded \(\Kring_p\)\nb-modules lifts to a
  \(\KK^{\Z/p}\)-equivalence in \(\KK^{\Z/p}_0(A,C)\).  In particular,
  \(A\) and~\(C\) are \(\KK^{\Z/p}\)-equivalent if and only if
  \(U_{\mathfrak{C}^{\Z/p}}(A) \cong U_{\mathfrak{C}^{\Z/p}}(C)\) as
  \(\Z/2\)\nb-graded \(\Kring_p\)\nb-modules.
\end{theorem}

\begin{proof}
  This follows from Theorem~\ref{the:range_U} and
  Theorem~\ref{thm:truct}.
\end{proof}

\section{Equivariant Kasparov groups for products of groups}
\label{sec:KK_product}

A cyclic group of square-free order is a product of cyclic groups of
prime order.  To deal with equivariant KK-groups for such groups, we
shall use the following proposition, which computes certain
equivariant KK-groups for products of groups.

\begin{definition}
  A \(\Cst\)\nb-algebra \(A\in\KK^G\) is \emph{Poincar\'e dualisable}
  if there is a \(\Cst\)\nb-algebra \(A'\in\KK^G\), called its
  \emph{Poincar\'e dual}, such that
  \(\KK^G(B\otimes A,C)\cong\KK^G(B,C\otimes A')\) for all
  \(B,C\in\KK^G\).
\end{definition}

\begin{proposition}
  \label{pro:isom}
  Let \(H\) and~\(G\) be finite groups.  Let
  \(A_1,A_2 \in \mathfrak{B}^H\) and \(B_1,B_2 \in \mathfrak{B}^G\).
  Assume that \(A_1\) and~\(B_1\) are Poincar\'e dualisable and that
  the abelian groups \(\KK^H(A_1,A_2)\) and \(\KK^G(B_1,B_2)\) are
  torsionfree.  Then the external tensor product map induces an
  isomorphism
  \[
    \KK^H(A_1,A_2)\otimes_\Z \KK^G(B_1,B_2)
    \xrightarrow{\cong}
    \KK^{H\times G}( A_1\otimes B_1, A_2\otimes B_2 ).
  \]
\end{proposition}

\begin{proof}
  Let \(A_1'\) and~\(B_1'\) be the Poincar\'e duals of \(A_1\)
  and~\(B_1\), respectively.  Duality gives natural isomorphisms
  \begin{align*}
    \pi_{A_1}\colon \KK^H(A_1,A_2)
    &\xrightarrow{\cong}\KK^H(\C,\,A_2\otimes A_1'),\\
    \pi_{B_1}\colon \KK^G(B_1,B_2)
    &\xrightarrow{\cong}\KK^G(\C,\,B_2\otimes B_1').
  \end{align*}
  Next we apply the Green--Julg isomorphisms for finite groups,
  \begin{align*}
    \mu_H\colon \KK^H(\C,\,A_2\otimes A_1')
    &\xrightarrow{\cong}
      \K_0\bigl((A_2\otimes A_1')\rtimes H\bigr),\\\
    \mu_G\colon \KK^G(\C,\,B_2\otimes B_1')
    &\xrightarrow{\cong}
      \K_0\bigl((B_2\otimes B_1')\rtimes G\bigr).
  \end{align*}
  It follows that
  \[
    \KK^H(A_1,A_2)\cong\K_0\bigl((A_2\otimes A_1')\rtimes H\bigr)\text{ and }
    \KK^G(B_1,B_2)\cong\K_0\bigl((B_2\otimes B_1')\rtimes G\bigr).
  \]

  We claim that the Poincar\'e dual of an object in the equivariant
  bootstrap class also belongs to the equivariant bootstrap class.
  This follows from \cite{dellAmbrogio:Cell_G}*{Proposition~2.9},
  where dualisable objects are called \emph{rigid}.  This is because
  any object that is dualisable in \(\KK^G\) must be compact, and by
  \cite{dellAmbrogio:Cell_G}*{Proposition~2.9} all those compact
  objects already have a dual in the bootstrap class.  The equivariant
  bootstrap class~\(\mathfrak{B}^G\) is closed under tensor products
  because all tensor products of the generators
  \(\Cont(G/H) \otimes \Cont(G/L)\) belong to~\(\mathfrak{B}^G\).
  Therefore, \(B_2\otimes B_1' \in \mathfrak{B}^G\).  This implies
  that the crossed product \((B_2\otimes B_1')\rtimes G\) belongs to
  the nonequivariant bootstrap class.  Similarly,
  \((A_2\otimes A_1')\rtimes H\) belongs to the bootstrap class.  Then
  the K\"unneth Theorem applies to the tensor product of these two
  crossed products.  Since the abelian groups \(\KK^H(A_1,A_2)\) and
  \(\KK^G(B_1,B_2)\) are torsionfree by assumption, the Tor-term in
  the K\"unneth formula vanishes.  Thus the following map is an
  isomorphism:
  \begin{multline*}
    \alpha\colon
    \K_0\bigl(( A_2\otimes A_1')\rtimes H\bigr)\otimes_\Z
    \K_0 \bigl(( B_2\otimes B_1')\rtimes G\bigr)
    \\ \to
    \K_0\Bigl(  \bigl(( A_2\otimes A_1')\rtimes H\bigr) \otimes
    \bigl(( B_2\otimes B_1')\rtimes G\bigr)\Bigr).
  \end{multline*}
  The universal property for crossed products implies that there is a
  canonical isomorphism
  \[
    \Phi\colon \bigl((A_2\otimes A_1')\rtimes H\bigr)\otimes
    \bigl((B_2\otimes B_1')\rtimes G\bigr)
    \xrightarrow{\cong}(A_2\otimes B_2\otimes A_1'\otimes B_1')\rtimes(H\times G),
  \]
  where the action of \(H\times G\) on \(A_2\otimes B_2\otimes
  A_1'\otimes B_1'\) is the diagonal one induced by the given actions.
  The inverse \(\mu^{-1}_{H\times G}\) of the Green--Julg isomorphism
  and the inverse Poincar\'e duality isomorphism
  \(\pi^{-1}_{A_1\otimes B_1}\) give isomorphisms
  \begin{multline*}
    \K_0\bigl(( A_2\otimes B_2 \otimes A_1'\otimes B_1')
    \rtimes (H\times G)\bigr)
    \cong \KK^{H\times G}(\C,A_2\otimes B_2 \otimes A_1'\otimes B_1')
    \\ \cong \KK^{H\times G}( A_1\otimes B_1, A_2\otimes B_2 ).
  \end{multline*}
  Summing up, we constructed an isomorphism
  \[
    \KK^H(A_1,A_2)\otimes_\Z \KK^G(B_1,B_2)
    \xrightarrow{\psi} \KK^{H\times G}( A_1\otimes B_1, A_2\otimes B_2 )
  \]
  with
  \(\psi =\pi^{-1}_{A_1\otimes B_1}\circ \mu^{-1}_{H\times G}\circ
  \Phi_*\circ \alpha\circ \bigl((\mu_H\circ\pi_{A_1})\otimes
  (\mu_G\circ\pi_{B_1})\bigr)\).  It remains to show that this map is
  the external tensor product map.  The external Kasparov product is
  compatible with Poincar\'e duality and the Green--Julg isomorphism; that
  is
  \begin{align*}
    \pi_{A_1\otimes B_1}(x\otimes_{\ext} y)
    &= \pi_{A_1}(x)\otimes_{\ext}\pi_{B_1}(y),\\
    \mu_{H\times G}\bigl(\xi\otimes_{\ext}\eta\bigr)
    &= \Phi_*\bigl(\mu_H(\xi)\otimes \mu_G(\eta)\bigr).
  \end{align*}
  Therefore,
  \begin{align*}
    (\mu_{H\times G}\circ \pi_{A_1\otimes B_1})(x \otimes_{\ext} y)
    &= \mu_{H\times G}(\pi_{A_1}(x) \otimes_{\ext} \pi_{B_1}(y)) \\
    &= \Phi_*(\alpha(\mu_H(\pi_{A_1}(x)) \otimes \mu_G(\pi_{B_1}(y))))
    \\ &= \Phi_*\circ \alpha\circ ((\mu_H\circ\pi_{A_1})
         \otimes (\mu_G\circ\pi_{B_1}))(x \otimes y).
  \end{align*}
  Using this identity, we substitute
  \begin{align*}
    \psi(x \otimes y)
    & =  \pi^{-1}_{A_1\otimes B_1}\circ \mu^{-1}_{H\times G}
      \circ \Phi_* \circ \alpha\circ ((\mu_H\circ\pi_{A_1})
      \otimes (\mu_G\circ\pi_{B_1}))(x \otimes y) \\
    & =  \pi^{-1}_{A_1\otimes B_1}\circ \mu^{-1}_{H\times G}
      \circ  (\mu_{H\times G}\circ \pi_{A_1\otimes B_1})
      (x \otimes_{\ext} y)
      = x\otimes_{\ext}y.\qedhere
  \end{align*}
\end{proof}

\section{Extension to cyclic groups of square-free order}
\label{sec:extension}

In this section, we prove a Universal Coefficient Theorem for a cyclic
group~\(G\) of square-free order by reducing to the case of cyclic
groups of prime order.
Fix a presentation
\[
  G\cong \Z/p_1\times\Z/p_2\times \dots \times \Z/p_k,
\]
where \(p_1, \dotsc, p_k\) are distinct prime numbers.

We define an invariant as follows.  For each prime~\(p_i\),
let~\(C_{p_i}\) denote the mapping cone of the \(G\)\nb-equivariant
unital embedding \(\C\to \Cont(\Z/p_i)\).  Define
\[
  \mathfrak{C}^G \defeq
  \setgiven[\bigl]{ A_1 \otimes A_2 \otimes \dots \otimes A_k}{A_j \in \{ C_{p_j},\ \C,\ \Cont(\Z/p_j) \}}.
\]
Here \(A_1 \otimes A_2 \otimes \dots \otimes A_k\) carries the action
of~\(G\) induced by the \(\Z/p_j\)\nb-action on~\(A_j\) for
\(j=1,\dotsc,k\).

If \(H \subseteq G\) is a cyclic subgroup, then
\(H= H_1 \times \dotsb \times H_k\) with \(H_j =\{1\}\) or
\(H_j = \Z/p_j\) for \(j=1,\dotsc,k\).  Then
\(\Cont(G/H) = A_1 \otimes \dotsb \otimes A_k\) with
\(A_j = \Cont(\Z/p_j)\) or \(A_j = \C\) for \(j=1,\dotsc,k\).  Thus
\(\Cont(G/H)\in \mathfrak{C}^G\) for all cyclic subgroups
\(H\subseteq G\).

Tensoring the exact triangle
\[
  \Sigma \Cont(\Z/p_j) \to C_{p_j}\to \C\to \Cont(\Z/p_j)
\]
with \(A_i\) as above shows that
\[
  \bigl\langle \setgiven[\bigl]{ A_1 \otimes A_2 \otimes \dots \otimes
    A_k}{A_i \in \{\C,\, \Cont(\Z/p_i) \}} \bigr\rangle
  = \langle \mathfrak{C}^G \rangle.
\]
Therefore, \(\mathfrak{C}^G\) generates the \(G\)\nb-equivariant
bootstrap class, \(\langle \mathfrak{C}^G \rangle = \mathfrak B^G\).

We denote the \(\Z/2\)\nb-graded category ring of
\((\mathfrak{C}^G)^\op\) by~\(\Kring_G\) or~\(\Kring_{p_1\dotsm
  p_k}\).
For \(G=\Z/p\), this is the ring introduced by Manuel K\"ohler (see
Section~\ref{sec:kohler}).

\begin{remark}
  It is crucial that~\(G\) has square-free order.
  If
  \[
    G\cong \Z/p_1\times\Z/p_2\times \dots \times \Z/p_k
  \]
  but the primes~\(p_j\) are not all distinct, then there is a
  subgroup~\(H\) that is not of product type as above.
  Then the generator \(\Cont(G/H)\) of~\(\mathfrak{B}^G\) does not belong
  to~\(\mathfrak{C}^G\).
  By Lemma~\ref{lem:K_does_not_detect}, the localising subcategory
  generated by the generators in~\(\mathfrak{C}^G\)
  is smaller than the equivariant bootstrap class.
  Therefore, the invariant~\(U_{\mathfrak{C}^G}\) cannot give a
  Universal Coefficient Theorem for~\(G\) that holds on the entire
  bootstrap class.
\end{remark}

\begin{corollary}
  \label{cor:isom}
  Let \(p_1,\dotsc, p_k\) be distinct primes, and let
  \(G \cong \Z/p_1 \times \dots \times \Z/p_k\).  The external tensor
  product induces an isomorphism of rings
  \(\bigotimes_{i=1}^k \Kring_{p_i}\cong \Kring_{p_1\dotsm p_k}\),
  where the tensor product is taken over the ring of integers.
\end{corollary}

\begin{proof}
  All objects in~\(\mathfrak{C}^G\) are Poincar\'e dualisable by
  \cite{dellAmbrogio:Cell_G}*{Proposition~2.9}, where these are
  referred to as \emph{rigid} objects.  Then,
  Proposition~\ref{pro:isom} applied \(k-1\) times implies, by
  induction, that the external tensor product induces a group
  isomorphism
  \(\bigotimes_{i=1}^k \Kring_{p_i}\cong \Kring_{p_1\dotsm p_k}\).
  This is even a ring isomorphism because the external tensor product
  preserves Kasparov products.
\end{proof}

We now generalise Definition~\ref{def:exact_K-module} to finite cyclic
groups of square-free order.

\begin{definition}
  \label{def:exact_module}
  Let~\(\MG\) be a left module over~\(\Kring_G\).
  The external tensor product map induces canonical maps
  \(\Kring_{p_m} \to \bigotimes_{i=1}^k \Kring_{p_i} \cong \Kring_G\)
  for every \(m=1,\dotsc,k\).
  We call~\(\MG\) \emph{exact} if it is exact as a
  \(\Kring_{p_m}\)\nb-module for \(m=1,\dotsc,k\).
\end{definition}

Any \(\Kring_G\)\nb-module of the form \(U_{\mathfrak{C}^G}(A)\) for
\(A\in\KK^G\) is exact: this follows from the corresponding statement
for cyclic groups of prime order.

\begin{lemma}
  \label{lem:ext1_van}
  Let~\(\MG\) be an exact left \(\Kring_G\)\nb-module that is countably
  generated and free as an abelian group.
  Then \(\Ext^1_{\Kring_{G}}(\MG,\MG')= 0\) for any left
  \(\Kring_{G}\)-module~\(\MG'\).
  Therefore, \(\MG\) is a projective module over~\(\Kring_G\).
\end{lemma}

\begin{proof}
  Recall that \(\Ext^1_{\Kring_{G}}(\MG,\MG')= 0\) for any left
  \(\Kring_{G}\)-module~\(\MG'\) if and only if~\(\MG\) is a
  projective \(\Kring_{G}\)-module.
  We prove this by induction on the number~\(k\) of prime factors
  of~\(\abs{G}\).
  The base case for a single prime was proven by Manuel K\"ohler,
  see~\cite{Koehler:Thesis}*{Theorem~12.4} or
  Theorem~\ref{the:range_U}.
  Let
  \[
    \sigma\colon 0\to \MG'\to \MG'' \xrightarrow{\beta} \MG\to 0
  \]
  be any extension of left
  \(\Kring_{p_1 \dotsm p_k}\)-modules.
  We must prove that it splits by a module homomorphism.
  Corollary~\ref{cor:isom} implies
  \[
    \Kring_{p_1\dotsm p_k}
    \cong \biggl( \bigotimes_{i=1}^{k-1} \Kring_{p_i} \biggr)
    \otimes_{\Z} \Kring_{p_k}
    \cong \Kring_{p_1\dotsm p_{k-1}}\otimes_{\Z} \Kring_{p_k}.
  \]
  By the induction hypothesis, \(\MG\) is projective over
  \(\Kring_{p_1 \dotsm p_{k-1}}\).
  Therefore, \(\sigma\) splits by a \(\Kring_{p_1 \dotsm
    p_{k-1}}\)\nb-linear map \(\gamma \colon \MG\to \MG''\).
  We are going to define another map \(\gamma\colon \MG\to \MG''\)
  that it is both \(\Kring_{p_k}\)- and \(\Kring_{p_1 \dotsm p_{k-1}}\)-linear.
  Then it is a module map over the whole ring \(\Kring_{p_1\dotsm
    p_{k-1}}\otimes_{\Z} \Kring_{p_k} \cong \Kring_{p_1 \dotsm p_k}\).
  In the following computations, we abbreviate \(p \defeq p_k\) and
  suppress some tensor products to avoid clutter.
  The generators \(\alpha_{i j}\), \(t_j\) and~\(s_j\) are those in
  the ring~\(\Kring_p\).
  We change~\(\gamma\) in three steps.

  The first step makes~\(\gamma\) compatible with the decomposition
  \(\MG= \MG_0\oplus \MG_1 \oplus \MG_2\), where \(\MG_j \defeq 1_j
  \MG\) with the idempotent generators of~\(\Kring_p\).
  The map~\(\gamma\) is represented by a \(3\times 3\)-matrix of maps
  \(1_l \gamma 1_j\colon \MG_j \to \MG_l''\) for \(j,l\in\{0,1,2\}\).
  As~\(\beta\) is \(\Kring_p\)\nb-linear,
  \[
    \beta\left(\sum_{i=1}^3 1_i\gamma 1_i\right)
    = \sum_{i=1}^3 1_i\beta\gamma 1_i
    = \sum_{i=1}^3 1_i \mathrm{id}_{\MG} 1_i
    = \mathrm{id}_\MG.
  \]
  So we may replace~\(\gamma\) by \(\sum_{i=1}^3 1_i\gamma 1_i\),
  which makes it diagonal.
  We assume from now on that~\(\gamma\) is diagonal and write
  \[
    \gamma = \gamma_0+\gamma_1+\gamma_2
  \]
  with maps \(\gamma_i\colon \MG_i \to \MG_i''\) for \(i=0,1,2\).

  In the second step, we define a new map~\(\gamma'\) that commutes
  with \(t_0\) and~\(s_1\).
  We achieve this by employing the relations \eqref{eq:tp}
  and~\eqref{eq:sp} to average over \(t_0\) and~\(s_1\) at the objects
  \(0\) and~\(1\), respectively (see Figure~\ref{fig:K_p}), and
  multiplying by~\(p\) at the object~\(2\).
  More precisely, let
  \[
    \gamma' \defeq
    \sum_{i=0}^{p-1}t^i_0\gamma_0 t^{p-i}_0
    + \sum_{i=0}^{p-1}s^i_1\gamma_1 s^{p-i}_1 + p \gamma_2.
  \]
  This map~\(\gamma'\) commutes with \(t_0\) and~\(s_1\), so that it
  also commutes with \(N(t_0)\) and~\(N(s_1)\).

  In the third step, we replace~\(\gamma'\) by a map~\(\gamma''\) that
  commutes with~\(\alpha_{ij}\) for all distinct \(i,j\in\{0,1,2\}\).
  The idea behind this is to think of~\(\alpha_{ij}\) as a partial
  isometry up to the constant~\(p\) and to identify the corresponding
  \(p\)\nb-projection elements in the ring.
  Note that
  \begin{align*}
    p\cdot 1_0 -N(t_0)
    &= (1_0-t_0)\sum_{i=0}^{p-2}(p-i-1)t_0^i,\\
    p\cdot 1_1 -N(s_1)
    &= (1_1-s_1)\sum_{i=0}^{p-2}(p-i-1)s_1^i.
  \end{align*}
  Then we define
  \begin{align*}
    \gamma''_0
    &\defeq p\cdot \gamma'_0,\\
    \gamma''_1
    &\defeq \alpha_{10}\gamma'_0 \alpha_{01}
      + (p\cdot 1_1-\alpha_{10}\alpha_{01})\gamma'_1,\\
    \gamma''_2
    &\defeq \alpha_{20}\gamma'_0
      \left(\sum_{i=0}^{p-2}(p-i-1)t_0^i\right)\alpha_{02}
      + \alpha_{21}\gamma'_1
      \left(\sum_{i=0}^{p-2}(p-i-1)s_1^i\right)\alpha_{12}.
  \end{align*}
  We claim that~\(\gamma''\) commutes with \(\alpha_{01}\)
  and~\(\alpha_{10}\).
  The relations \eqref{eq:alpha_010_N}, \eqref{eq:alpha_101_N},
  \eqref{eq:ta_01}, and \eqref{eq:sa_10} imply
  \begin{align*}
    \alpha_{01}\alpha_{10}\alpha_{01}
    &= N(t_0)\alpha_{01}= p\alpha_{01},\\
    \alpha_{10}\alpha_{01}\alpha_{10}
    &= N(s_1)\alpha_{10}= p\alpha_{10}.
  \end{align*}
  Therefore,
  \begin{multline*}
    \gamma''\alpha_{10}
    = \gamma''_1 \alpha_{10}
    = \alpha_{10}\gamma'_0\alpha_{01}\alpha_{10}
    + p\gamma'_1\alpha_{10}
    - \alpha_{10}\alpha_{01}\gamma_1'\alpha_{10}
    \\= \alpha_{10}\gamma'_0N(t_0)+p\gamma'_1\alpha_{10}
    - N(s_1)\gamma'_1\alpha_{10}
    = \alpha_{10}\alpha_{01}\alpha_{10}\gamma'_0
    + p\gamma'_1\alpha_{10}
    - \gamma'_1\alpha_{10}\alpha_{01}\alpha_{10}
    \\= p\alpha_{10}\gamma'_0  +p\gamma'_1\alpha_{10}
    - p\gamma'_1\alpha_{10}
    = p\alpha_{10}\gamma'_0
    = \alpha_{10}\gamma''_0
    = \alpha_{10}\gamma''.
  \end{multline*}
  Similar computations show that
  \(\alpha_{01}\gamma'' = \gamma''\alpha_{01}\) and
  that~\(\gamma''\) commutes with~\(s_1\) and~\(t_0\).

  Next, using the relations \eqref{eq:alpha_020},
  \eqref{eq:alpha_121}, \eqref{eq:Na_02}, \eqref{eq:aN_20},
  \eqref{eq:Na_12}, and~\eqref{eq:aN_21} and that
  \(\alpha_{mk}\alpha_{kj}=0\) for \(\{j,k,m\}=\{0,1,2\}\), we
  compute
  \begin{multline*}
    \gamma''\alpha_{21}
    = \gamma''_2\alpha_{21}
    = \alpha_{21}\gamma'_1
    \left(\sum_{i=0}^{p-2}(p-i-1)s_1^i\right)\alpha_{12}\alpha_{21}
    \\= \alpha_{21}\gamma'_1
    \left(\sum_{i=0}^{p-2}(p-i-1)s_1^i\right)(1_1-s_1)
    = \alpha_{21}\gamma'_1(p\cdot 1_1-N(s_1))
    \\= p\alpha_{21}\gamma'_1 - \alpha_{21}N(s_1)\gamma'_1
    = p\alpha_{21}\gamma'_1
    = \alpha_{21}\gamma''_1
    = \alpha_{21}\gamma'',
  \end{multline*}
  \begin{multline*}
    \gamma''\alpha_{12}
    = \gamma''_1\alpha_{12}
    = \alpha_{10}\gamma'_0 \alpha_{01}\alpha_{12}
    + (p\cdot 1_1-\alpha_{10}\alpha_{01})\gamma'_1\alpha_{12}
    \\= p\gamma'_1\alpha_{12}
    - \gamma'_1\alpha_{10}\alpha_{01}\alpha_{12}
    = p\gamma'_1\alpha_{12}
    = \gamma'_1 (p\cdot1_1-N(s_1))\alpha_{12}
    \\= \alpha_{12}\alpha_{21}\gamma'_1
    \left(\sum_{i=0}^{p-2}(p-i-1)s_1^i\right)\alpha_{12}
    = \alpha_{12}\gamma''_2
    = \alpha_{12}\gamma''.
  \end{multline*}
  Similarly, \(\gamma''\alpha_{20} = \alpha_{20}\gamma''\) and
  \(\gamma''\alpha_{02} = \alpha_{02}\gamma''\).  As the
  elements~\(\alpha_{ij}\) generate the ring~\(\Kring_p\), it
  follows that~\(\gamma''\) is \(\Kring_p\)-linear.  Since it
  remains \(\Kring_{p_1\dotsm p_{k-1}}\)-linear, it is
  \(\Kring_{p_1\dotsm p_{k-1} p}\)-linear.

  The map~\(\gamma''\) is not a section for~\(\sigma\).  Instead,
  \[
    \beta\gamma'
    = \sum_{i=0}^{p-1}t^i_0\beta 1_0\gamma t^{p-i}_0
    + \sum_{i=0}^{p-1}s^i_1\beta 1_1\gamma s^{p-i}_1
    + p \beta1_2 \gamma
    = p\cdot(1_0+1_1+1_2)
    =  p\cdot 1_\MG
  \]
  and
  \begin{multline*}
    \beta\gamma''
    = \beta(\gamma_0''+\gamma_1''+\gamma_2'')
    =  p\beta \gamma'_0 +\alpha_{10}\beta \gamma'_0 \alpha_{01}
    + (p\cdot 1_1-\alpha_{10}\alpha_{01})\beta\gamma'_1
    \\ +\alpha_{20}\beta\gamma'_0
    \left(\sum_{i=0}^{p-2}(p-i-1)t_0^i\right)\alpha_{02}
    + \alpha_{21}\beta\gamma'_1
    \left(\sum_{i=0}^{p-2}(p-i-1)s_1^i\right)\alpha_{12}
    \\= p^2\cdot 1_0 + p^2\cdot 1_1
    + p(p\cdot 1_2-N(t_2))+ p(p\cdot 1_2-N(s_2))
    \\= p^2\cdot 1_0 + p^2\cdot 1_1 + p^2\cdot 1_2
    = p^2\cdot 1_\MG.
  \end{multline*}
  Here we used the relations \eqref{eq:alpha_2_N},
  \eqref{eq:alpha_202}, \eqref{eq:alpha_212}, \eqref{eq:ta_20}
  and~\eqref{eq:sa_21}.

  The above calculations say that~\(\gamma''\) is a section for the
  extension~\(p^2\sigma\).
  Thus \([p^2\sigma]=0\) in the group \(\Ext^1_{\Kring_{p_1\dotsm
      p_{k-1} p}}(\MG,\MG')\).
  The same argument may be used for~\(p_1\) instead of~\(p\).
  This shows \([p_1^2\sigma]=0\).
  Since \(p\) and~\(p_{1}\) are coprime, there are \(m,n\in\Z\) with
  \(m p^2 + n p_1^2 =1\).
  So
  \[
    [\sigma] = 1[\sigma] = m[p^2\sigma]+n [p_1^2\sigma] = 0
    \quad
    \text{in }\Ext^1_{\Kring_{p_1\dotsm p_{k-1}p}}(\MG,\MG').\qedhere
  \]
\end{proof}

\begin{proposition}
  \label{pro:range-U_G}
  Let~\(G\) be a finite cyclic group of square-free order.  Let~\(\MG\)
  be a countable \(\Z/2\)\nb-graded left \(\Kring_G\)\nb-module.  The
  following are equivalent:
  \begin{enumerate}[label=\textup{(\arabic*)}]
  \item \(\MG = U_{\mathfrak{C}^{G}}(A)\) for some~\(A\)
    in~\(\mathfrak{B}^{G}\);
  \item \(\MG = U_{\mathfrak{C}^{G}}(A)\) for some~\(A\) in~\(\KK^{G}\);
  \item \(\MG\) is exact;
  \item \(\MG\) has a projective \(\Kring_G\)\nb-module resolution of
    length~\(1\).
  \end{enumerate}
\end{proposition}

\begin{proof}
  The implication (1)\(\Rightarrow\)(2) is trivial,
  (2)\(\Rightarrow\)(3) is an easy consequence of the definition
  of~\(\alpha_{i j}\), and (3)\(\Rightarrow\)(4) is
  Lemma~\ref{lem:ext1_van}.
  The implication (4)\(\Rightarrow\)(1) is a general statement in the
  setting of homological algebra in triangulated categories; see, for
  example, \cite{Bentmann-Meyer:More_general}*{Proposition~2.3}.
\end{proof}

Finally, Theorem~\ref{the:Eqv_UCT} follows from
Theorem~\ref{thm:truct} because Proposition~\ref{pro:range-U_G} gives
the required length-\(1\) projective resolutions.

\section{Uniquely divisible modules}
\label{sec:uniquely_p-divisible}

We now prove stronger results when a \(\Kring_G\)\nb-module is
uniquely divisible for some of the primes dividing~\(\abs{G}\).

Recall that for a prime~\(p\), a module~\(\MG\) is called
\emph{uniquely~\(p\)\nb-divisible} if multiplication by~\(p\) on~\(\MG\)
is invertible.
For a \(\Kring_G\)\nb-module~\(\MG\), a prime~\(p\) dividing the order
of~\(G\), and \(i \in \{0, 1, 2\}\), let~\(\MG_i^{(p)}\) denote the
direct summand \(1^{(p)}_i\cdot \MG\) defined by the projection
\(1^{(p)}_i\in\Kring_p\).
This is still a module over \(\Kring_{\abs{G}/p}\), the tensor product
of~\(\Kring_{p_j}\) for the other prime divisors of~\(\abs{G}\), and
\(\MG = \MG^{(p)}_0 + \MG^{(p)}_1 + \MG^{(p)}_2\).
By~\(\vartheta_p\) we denote a primitive \(p\)th root of unity.

\begin{proposition}
  \label{pro:div_modules}
  Let~\(G\) be a finite cyclic group of square-free order \(\abs{G} =
  p_1 \cdot p_2 \dotsm p_n\) with distinct primes \(p_1,\dotsc,p_n\).
  Let~\(\DmG\) be an exact \(\Kring_G\)\nb-module.
  Let \(1\le k \le n\).
  Suppose that for each prime \(p\in\{p_1,\dotsc, p_k\}\),
  one of the components \(\DmG_0^{(p)}\), \(\DmG_1^{(p)}\) or~\(\DmG_2^{(p)}\)
  is uniquely \(p\)\nb-divisible.
  Then, for each such prime, the remaining two components are also
  uniquely \(p\)\nb-divisible.

  For \(I\in \{X,Y,Z\}^k\), let \(R_I = \Z[1/p_1p_2 \dotsm p_k]
  [\setgiven{\vartheta_{p_j}}{I_j\in\{Y,Z\}}]\).
  There are \(3^k\) \(\Z/2\)\nb-graded exact \(\Kring_{p_{k+1}\dotsm
    p_n} \otimes_\Z R_I\)\nb-modules \((A_I)_{I\in \{X,Y,Z\}^k}\) that
  together determine the module~\(\DmG\) uniquely, in a way explained
  during the proof.
  If \(k=n\), then~\(A_I\) is simply an \(R_I\)\nb-module.
\end{proposition}

\begin{proof}
  We fix \(p\in\{p_1,\dotsc,p_k\}\) and focus on the
  \(\Kring_p\)\nb-module structure on~\(\DmG\) for a moment.
  By~\cite{Meyer:Actions_Kirchberg}*{Theorem~7.2}, the pieces
  \(\DmG_0^{(p)}\), \(\DmG_1^{(p)}\) and~\(\DmG_2^{(p)}\) are uniquely
  \(p\)\nb-divisible.
  There are a \(\Z/2\)\nb-graded \(\Z[1/p]\)\nb-module \(A_X\) and
  \(\Z/2\)\nb-graded \(\Z[1/p, \vartheta_{p}]\)\nb-modules
  \(A_Y,A_Z\) such that~\(\DmG\) has the following form as a
  \(\Kring_{p}\)\nb-module:
  \begin{alignat*}{3}
    \DmG_0^{(p)}
    &= A_X\oplus A_Y,
    &\ 
      \DmG_1^{(p)}
    &= A_X\oplus A_Z,
    &\ 
      \DmG_2^{(p)}
    &= A_Y\oplus \Sigma A_Z,\\
    \Ga{0}{1}^{\DmG^{(p)}}
    &= \begin{pmatrix} 1^{A_X}&0\\0&0 \end{pmatrix},
    &\ 
      \Ga{1}{2}^{\DmG^{(p)}}
    &= \begin{pmatrix} 0&0\\0&(1-\vartheta_{p})^{A_Z} \end{pmatrix},
    &\ 
      \Ga{2}{0}^{\DmG^{(p)}}
    &= \begin{pmatrix} 0&1^{A_Y}\\0&0 \end{pmatrix},\\
    \Ga{1}{0}^{\DmG^{(p)}}
    &= \begin{pmatrix} {p}^{A_X}&0\\0&0 \end{pmatrix},
    &\ 
      \Ga{2}{1}^{\DmG^{(p)}}
    &= \begin{pmatrix} 0&0\\0&1^{A_Z} \end{pmatrix},
    &\ 
      \Ga{0}{2}^{\DmG^{(p)}}
    &= \begin{pmatrix} 0&0\\(1-\vartheta_{p})^{A_Y}&0 \end{pmatrix}.
  \end{alignat*}
  Here~\(\Sigma A_Z\) means~\(A_Z\) with opposite parity.
  Also,
  \begin{alignat*}{2}
    N(\Gt{0}^{\DmG^{(p)}})
    &= \begin{pmatrix} p&0\\0&0 \end{pmatrix},
    &\qquad N(\Gs{1}^{\DmG^{(p)}})
    &= \begin{pmatrix} p&0\\0&0 \end{pmatrix},\\
    N(\Gt{2}^{\DmG^{(p)}})
    &= \begin{pmatrix} 0&0\\0&p \end{pmatrix},
    &\qquad N(\Gs{2}^{\DmG^{(p)}})
    &= \begin{pmatrix} p&0\\0&0 \end{pmatrix}.
  \end{alignat*}
  We denote this module by \(\DmG^{(p)}(A_X,A_Y,A_Z)\).
  The following elements of \(\Kring_p[1/p]\) are central
  idempotents, projecting onto \(\DmG^{(p)}(A_X,0,0)\),
  \(\DmG^{(p)}(0,A_Y,0)\) and \(\DmG^{(p)}(0,0,A_Z)\),
  respectively:
  \begin{align*}
    E_X^{\DmG^{(p)}}
    &\defeq \frac{1}{p}N(t_0^{\DmG^{(p)}})
      + \frac{1}{p}N(s_1^{\DmG^{(p)}}),\\
    E_Y^{\DmG^{(p)}}
    &\defeq
      \left(1_0^{\DmG^{(p)}} - \frac{1}{p}N(t_0^{\DmG^{(p)}})\right)
      + \frac{1}{p}N(s_2^{\DmG^{(p)}}),\\
    E_Z^{\DmG^{(p)}}
    &\defeq \left(1_1^{\DmG^{(p)}} -
      \frac{1}{p}N(s_1^{\DmG^{(p)}})\right)
      + \frac{1}{p}N(t_2^{\DmG^{(p)}}).
  \end{align*}
  They are well defined because of the unique \(p\)\nb-divisibility
  assumption.

  Now we consider again the entire \(\Kring_G\)\nb-module structure
  on~\(\DmG\).
  By the tensor product decomposition of Corollary~\ref{cor:isom},
  the actions of \(\Kring_{p_i}\) and~\(\Kring_{p_j}\) on~\(\DmG\)
  commute for all \(i,j\in\{1,\dotsc,n\}\).
  Each \(E_I^{\DmG^{(p_j)}}\) belongs to the centre of
  \(\Kring_G[1/p_j]\) and so these are commuting idempotent
  \(\Kring_G\)\nb-module endomorphisms of~\(\DmG\).
  Thus,
  \begin{equation}
    \label{eq:comp_idem}
    E_X^{\DmG^{(p_j)}} + E_Y^{\DmG^{(p_j)}} + E_Z^{\DmG^{(p_j)}}
    = 1_0^{\DmG^{(p_j)}} +1_1^{\DmG^{(p_j)}}+1_2^{\DmG^{(p_j)}}
    = 1_{\DmG^{(p_j)}(A_X,A_Y,A_Z)}
  \end{equation}
  by~\eqref{eq:alpha_2_N}, and
  \[
    \DmG^{(p_j)}(A_X,A_Y,A_Z)
    = \DmG^{(p_j)}(A_X,0,0) \oplus \DmG^{(p_j)}(0,A_Y,0)
    \oplus \DmG^{(p_j)}(0,0,A_Z)
  \]
  is a \(\Kring_G\)\nb-module isomorphic to~\(\DmG\), where the action of
  \(\Kring_{p_j}\) is determined by~\(\alpha_{i k}^{\DmG^{(p_j)}}\).

  Since \(A_X = \im(\alpha_{01}^{\DmG^{(p_j)}})\), \(A_Y =
  \im(\alpha_{02}^{\DmG^{(p_j)}})\), and
  \(A_Z=\im(\alpha_{12}^{\DmG^{(p_j)}})\), these components are
  \(\Kring_{p_i}\)\nb-modules for all \(p_i\neq p_j\) by
  Corollary~\ref{cor:isom}.
  Also, since  \(A_X\), \(A_Y\) and~\(A_Z\) are direct
  \(\Kring_{p_i}\)\nb-module summands of \(\DmG^{(p_j)}(A_X,0,0)\),
  \(\DmG^{(p_j)}(0,A_Y,0)\) and \(\DmG^{(p_j)}(0,0,A_Z)\), respectively,
  they are themselves exact \(\Kring_{p_i}\)\nb-modules.
  Of course, \(A_X\), \(A_Y\), \(A_Z\) completely determine
  \(\DmG^{(p_j)}(A_X,0,0)\), \(\DmG^{(p_j)}(0,A_Y,0)\) and
  \(\DmG^{(p_j)}(0,0,A_Z)\).

  When we iterate the decomposition process above for all the
  divisible primes \(p_1, p_2,\dots ,p_k\), we get~\(3^k\) modules
  \[
    \DmG_I \defeq
    \left(\bigotimes_{j=1}^k E_{I_j}^{\DmG^{(p_j)}}\right) \DmG
  \]
  for \(I\in\{X,Y,Z\}^k\).
  Here~\(\DmG_I\) is determined by a  \(\Z/2\)\nb-graded module~\(A_I\) over
  the ring \(\Kring_{p_{k+1}\dots p_n} \otimes_\Z R_I\).
  Specifically, 
  \[
    A_I = \left(\prod_{j=1}^k 1_{c(I_j)}^{(p_j)}\right)\DmG_I
    = \bigcap_{j=1}^k(\DmG_I)_{c(I_j)}^{(p_j)},
  \]
  for \(c(X)=0,\) \(c(Y)=0\), and \(c(Z)=1\).

  By~\eqref{eq:comp_idem}, there is an isomorphism of
  \(\Kring_G\)\nb-modules
  \[
    \bigoplus_{I\in\{X,Y,Z\}^k} \DmG_I
    = \bigoplus_{I\in\{X,Y,Z\}^k} \left(\bigotimes_{j=1}^k
      E_{I_j}^{\DmG^{(p_j)}}\right) \DmG\cong \DmG.  \qedhere
  \]
\end{proof}

We now consider the special case when \(G\cong \Z/p \times \Z/q\) for
two different primes \(p,q\).
Let~\(\DmG\) be a uniquely \(q\)\nb-divisible \(\Kring_G\)\nb-module.
By Proposition~\ref{pro:div_modules}, \(\DmG\) may be specified by a
triple of exact \(\Z/2\)\nb-graded \(\Kring_p\)\nb-modules \((A_X,A_Y,A_Z)\),
where~\(A_X\) is also a \(\Z[1/q]\)\nb-module and \(A_Y, A_Z\) are
modules over \(\Z[1/q, \vartheta_q]\).
Then, as shown in the proof of Proposition~\ref{pro:div_modules},
\[
  \DmG \cong \bigl(E_X^{\DmG^{(q)}}+ E_Y^{\DmG^{(q)}}+ E_Z^{\DmG^{(q)}}\bigr)\DmG
  = \DmG^{(q)}(A_X,A_Y,A_Z)
\]
as \(\Kring_G\)\nb-modules.

One way to get candidates for \(A_X\), \(A_Y\) and~\(A_Z\) is to
take any exact \(\Z/2\)\nb-graded \(\Kring_p\)\nb-modules \(A_X^0\),
\(A_Y^0\) and~\(A_Z^0\) and form
\[
  A_X \defeq A_X^0 \otimes_\Z \Z[1/q],\qquad
  A_Y \defeq A_Y^0 \otimes_\Z \Z[1/q,\vartheta_q],\qquad
  A_Z \defeq A_Z^0 \otimes_\Z \Z[1/q,\vartheta_q].
\]
Since \(\Z[1/q]\) and \(\Z[1/q,\vartheta_q]\) are flat, these are also
exact \(\Z/2\)\nb-graded \(\Kring_p\)\nb-modules, and they now gain
the extra structure that allows us to turn them into uniquely
\(q\)\nb-divisible \(\Kring_G\)\nb-modules.
In particular, we may let~\(A_I^0\) be one of the examples in
\cite{Meyer:Actions_Kirchberg}*{Section~9}.
These examples are important because, up to an extension by a uniquely
\(p\)\nb-divisible exact module, they give all examples of exact
\(\Z/2\)\nb-graded \(\Kring_p\)\nb-modules~\(\Mp\) where the
ingredient~\(\Mp_1\) is a cyclic group.
Except for Examples 9.2 and~9.4 in~\cite{Meyer:Actions_Kirchberg}, we
have \(\Mp_1 = \Z/p^k\) for some integer \(k\ge1\), so that \(\Mp_1 \otimes_\Z
\Z[1/q] \cong \Mp_1\).
This means that if~\(A_X^0\) is one of these examples, then \(A_X^0
\otimes_\Z \Z[1/q]\) also corresponds to a group action on a Cuntz
algebra.

In contrast, tensoring with \(\Z[\vartheta_q,1/q]\) always makes the
groups bigger, so that \(\Mp_1\otimes_\Z \Z[\vartheta_q,1/q]\) cannot be
cyclic any more.
Sometimes, we may instead find the required action of
\(\Z[1/q,\vartheta_q]\) on \(A_Y^0\) or~\(A_Z^0\) itself:

\begin{example}
  \label{exa:previous_examples_checked}
  In Examples 9.7, 9.9, 9.11, and 9.12
  in~\cite{Meyer:Actions_Kirchberg}, the pieces \(\Mp_0\), \(\Mp_1\)
  and~\(\Mp_2\) in the \(\Kring_p\)\nb-module~\(\Mp\) are all purely
  \(p\)\nb-torsion.
  This makes~\(\Mp\) a module over the ring~\(\Z_p\)
  of \(p\)\nb-adic integers in a natural way.
  By naturality, the \(\Z_p\)\nb-module structure commutes with the
  action of the ring~\(\Kring_p\), making~\(\Mp\) a module over \(\Kring_p
  \otimes_\Z \Z_p\).
  If there is a homomorphism \(\Z[1/q,\vartheta_q] \to \Z_p\), then we
  may use this to define an action of \(\Z[1/q,\vartheta_q] \otimes_\Z
  \Kring_p\) on~\(\Mp\), which then allows us to take~\(\Mp\) for the
  pieces \(A_Y\) or~\(A_Z\) in our decomposition.
  Such a homomorphism exists if and only if \(p\neq q\) (which we
  assume anyway) and \(q \mid p-1\).
  The latter ensures that there is a primitive \(q\)th root of unity
  in the multiplicative group of the field~\(\Z/p\) and hence also
  in~\(\Z_p\) by Hensel's Lemma.
\end{example}

When~\(\DmG\) is both uniquely \(p\)- and \(q\)\nb-divisible, then it is
\(\abs{G}\)\nb-divisible.
This suffices to apply the UCT proven in
\cite{Meyer-Nadareishvili:UCT_actions}*{Theorem~1.1}.
We also get this information in a more concrete way from
Proposition~\ref{pro:div_modules}.
Namely, we may specify~\(\DmG\) by \(9\) \(\Z/2\)\nb-graded abelian
groups
\[
  A_{(X,X)}, A_{(X,Y)}, A_{(X,Z)}, A_{(Y,X)}, A_{(Y,Y)}, A_{(Y,Z)}, A_{(Z,X)}, A_{(Z,Y)}, A_{(Z,Z)},
\]
where \(A_{(X,X)}\) is a module over \(\Z[1/pq]\), \(A_{(X,Y)}\) and
\(A_{(X,Z)}\) are modules over \(\Z[1/pq, \vartheta_q]\),
\(A_{(Y,X)}\) and \(A_{(Z,X)}\) are modules over \(\Z[1/pq,
\vartheta_p]\), whereas \(A_{(Y,Y)}\), \(A_{(Y,Z)}\), \(A_{(Z,Y)}\),
and \(A_{(Z,Z)}\) are modules over \(\Z[1/pq, \vartheta_p,
\vartheta_q]\).

Separating the exact \(\Kring_G\)\nb-module into its \(9\) independent
direct summands over~\(\Kring_q\), we get:
\begin{alignat*}{3}
  \DmG &\cong
  & \Bigl( & \phantom{{}+{}} E_{X}^{\DmG^{(p)}} \otimes E_{X}^{\DmG^{(q)}}
        + E_{X}^{\DmG^{(p)}} \otimes E_{Y}^{\DmG^{(q)}}
        + E_{X}^{\DmG^{(p)}} \otimes E_{Z}^{\DmG^{(q)}}  \\
    & &&+ E_{Y}^{\DmG^{(p)}} \otimes E_{X}^{\DmG^{(q)}}
        + E_{Y}^{\DmG^{(p)}} \otimes E_{Y}^{\DmG^{(q)}}
        + E_{Y}^{\DmG^{(p)}} \otimes E_{Z}^{\DmG^{(q)}} \\
    & &&+ E_{Z}^{\DmG^{(p)}} \otimes E_{X}^{\DmG^{(q)}}
        + E_{Z}^{\DmG^{(p)}} \otimes E_{Y}^{\DmG^{(q)}}
        + E_{Z}^{\DmG^{(p)}} \otimes E_{Z}^{\DmG^{(q)}}\Bigr) \DmG  \\
      &\cong &&\phantom{{}\oplus{}} \DmG^{(q)}(A_{(X,X)},0,0) \oplus
                   \DmG^{(q)}(0,A_{(X,Y)},0)
                   \oplus \DmG^{(q)}(0,0,A_{(X,Z)}) \\
    & &&\oplus \DmG^{(q)}(A_{(Y,X)},0,0) \oplus \DmG^{(q)}(0,A_{(Y,Y)},0)
           \oplus \DmG^{(q)}(0,0,A_{(Y,Z)}) \\
    & &&\oplus \DmG^{(q)}(A_{(Z,X)},0,0) \oplus \DmG^{(q)}(0,A_{(Z,Y)},0)
           \oplus \DmG^{(q)}(0,0,A_{(Z,Z)}).
\end{alignat*}


\bib{Arano-Kubota:Atiyah-Segal}{article}{
  author={Arano, Yuki},
  author={Kubota, Yosuke},
  title={A categorical perspective on the Atiyah--Segal completion theorem in KK-theory},
  journal={J. Noncommut. Geom.},
  volume={12},
  date={2018},
  number={2},
  pages={779--821},
  issn={1661-6952},
  review={\MR {3825201}},
  doi={10.4171/JNCG/291},
}

\bib{Bentmann-Meyer:More_general}{article}{
  author={Bentmann, Rasmus},
  author={Meyer, Ralf},
  title={A more general method to classify up to equivariant KK\nobreakdash -equivalence},
  journal={Doc. Math.},
  volume={22},
  date={2017},
  pages={423--454},
  issn={1431-0635},
  doi={10.4171/DM/570},
  review={\MR {3628788}},
}

\bib{Bouc-DellAmbrogio-Martos:Splitting}{article}{
  author={Bouc, Serge},
  author={Dell'Ambrogio, Ivo},
  author={Martos, Rub\'en},
  title={A general Greenlees--May splitting principle},
  journal={M\"unster J. Math.},
  volume={18},
  date={2025},
  number={1},
  pages={227--243},
  issn={1867-5778},
  doi={10.17879/11958511582},
  review={\MR {4976179}},
}

\bib{dellAmbrogio:Cell_G}{article}{
  author={Dell'Ambrogio, Ivo},
  title={Equivariant Kasparov theory of finite groups via Mackey functors},
  journal={J. Noncommut. Geom.},
  volume={8},
  date={2014},
  number={3},
  pages={837--871},
  issn={1661-6952},
  review={\MR {3261603}},
  doi={10.4171/JNCG/172},
}

\bib{Gabe-Szabo:Dynamical_Kirchberg}{article}{
  author={Gabe, James},
  author={Szab\'o, G\'abor},
  title={The dynamical Kirchberg--Phillips theorem},
  journal={Acta Math.},
  volume={232},
  date={2024},
  number={1},
  pages={1--77},
  issn={0001-5962},
  review={\MR {4747811}},
  doi={10.4310/ACTA.2024.v232.n1.a1},
}

\bib{Koehler:Thesis}{thesis}{
  author={K\"ohler, Manuel},
  title={Universal coefficient theorems in equivariant KK-theory},
  institution={Georg-August-Universit\"at G\"ottingen},
  type={phdthesis},
  date={2010},
  doi={10.53846/goediss-2561},
}

\bib{Meyer:Actions_Kirchberg}{article}{
  author={Meyer, Ralf},
  title={On the classification of group actions on \(\mathrm {C}^*\)-algebras up to equivariant KK-equivalence},
  journal={Ann. K-Theory},
  issn={2379-1683},
  date={2021},
  volume={6},
  number={2},
  pages={157--238},
  doi={10.2140/akt.2021.6.157},
  review={\MR {4301904}},
}

\bib{Meyer-Nadareishvili:UCT_actions}{article}{
  author={Meyer, Ralf},
  author={Nadareishvili, George},
  title={A universal coefficient theorem for actions of finite groups on C*-algebras},
  journal={J. Math. Sci. Univ. Tokyo},
  volume={33},
  date={2026},
  number={1},
  pages={21--47},
  issn={1340-5705},
  review={\ZBmath {08153048}},
  year={2026},
  eprint={https://www.ms.u-tokyo.ac.jp/journal/number/jms3301.html},
}

\bib{Meyer-Nest:BC}{article}{
  author={Meyer, Ralf},
  author={Nest, Ryszard},
  title={The Baum--Connes conjecture via localisation of categories},
  journal={Topology},
  volume={45},
  date={2006},
  number={2},
  pages={209--259},
  issn={0040-9383},
  review={\MR {2193334}},
  doi={10.1016/j.top.2005.07.001},
}

\bib{Meyer-Nest:Homology_in_KK}{article}{
  author={Meyer, Ralf},
  author={Nest, Ryszard},
  title={Homological algebra in bivariant $K$\nobreakdash -theory and other triangulated categories. I},
  conference={ title={Triangulated categories}, },
  book={ series={London Math. Soc. Lecture Note Ser.}, editor={Holm, Thorsten}, editor={J\o rgensen, Peter}, editor={Rouqier, Rapha\"el}, volume={375}, publisher={Cambridge Univ. Press}, place={Cambridge}, },
  date={2010},
  pages={236--289},
  review={\MR {2681710}},
  doi={10.1017/CBO9781139107075.006},
}

\bib{Meyer-Nest:Filtrated_K}{article}{
  author={Meyer, Ralf},
  author={Nest, Ryszard},
  title={\(C^*\)\nobreakdash -Algebras over topological spaces: filtrated \(\textup K\)\nobreakdash -theory},
  journal={Canad. J. Math.},
  volume={64},
  pages={368--408},
  date={2012},
  review={\MR {2953205}},
  doi={10.4153/CJM-2011-061-x},
}

\bib{Rosenberg-Schochet:Kunneth}{article}{
  author={Rosenberg, Jonathan},
  author={Schochet, Claude},
  title={The K\"unneth theorem and the universal coefficient theorem for equivariant \(\textup K\)\nobreakdash -theory and \(\textup {KK}\)-theory},
  journal={Mem. Amer. Math. Soc.},
  volume={62},
  date={1986},
  number={348},
  issn={0065-9266},
  review={\MR {0849938}},
  doi={10.1090/memo/0348},
}

\bib{Thevenaz-Webb:The_structure_of_Mackey_functors}{article}{
  author={Th{\'e}venaz, Jacques},
  author={Webb, Peter},
  issn={0002-9947},
  issn={1088-6850},
  doi={10.2307/2154915},
  review={\ZBmath {0834.20011}},
  title={The structure of Mackey functors},
  journal={Transactions of the American Mathematical Society},
  volume={347},
  number={6},
  pages={1865--1961},
  date={1995},
  publisher={American Mathematical Society (AMS), Providence, RI},
  eprint={infoscience.epfl.ch/record/130448},
}



\begin{bibdiv}
  \begin{biblist}
    \bibselect{references}
  \end{biblist}
\end{bibdiv}
\end{document}